\documentclass[12pt]{article}
\usepackage{geometry}
\usepackage{amsmath,amsthm, amssymb, latexsym}
\usepackage{graphicx}
\usepackage{mathtools}
\usepackage{varwidth}
\usepackage[hidelinks]{hyperref}
\usepackage{float} 
\usepackage{url}

\newtheorem{theorem}{Theorem}[section]
\newtheorem{proposition}{Proposition}[section]
\newtheorem{definition}{Definition}[section]

\newtheorem{corollary}{Corollary}[section]

\newtheorem{remark}{Remark}[section]

\begin{document}

\title{\Large Association Scheme on Triples from the Unitary Group}

\date{}

\author{\normalsize JOSE MARIA P. BALMACEDA\\
\small Institute of Mathematics\\ 
 \small University of the Philippines Diliman\\
\small E-mail: \texttt{jpbalmaceda@up.edu.ph}\\ \\ 
\normalsize DOM VITO A. BRIONES\\
\small Institute of Mathematics\\ 
\small University of the Philippines Diliman\\
\small E-mail: \texttt{dvbriones@math.upd.edu.ph}\\ \\
\normalsize JORIS N. BULORON\(^*\)\\
\small Mathematics Department\\ 
\small Cebu Normal University\\
\small E-mail: \texttt{jorisbuloron@yahoo.com}\\
\small \(^*\) Corresponding Author}

\maketitle
\begin{center}
{\bf Abstract}
\end{center}
\indent \indent 
An association scheme on triples (AST) is a three-dimensional analogue of a classical association scheme. 
Since much of the algebraic and combinatorial properties of ASTs remain unknown, it is natural to determine the structure constants of ASTs that are constructed similarly to classical counterparts.
Similar to how a transitive group action produces a Schurian classical association scheme, a two-transitive group action produces an AST.
The parameters of several such ASTs are known; however, the intersection numbers of the ASTs from the finite unitary group remained undetermined.
In this paper, the parameters of the AST from the finite unitary group are obtained through explicit descriptions of the AST's relations and certain equations on the underlying field.
In particular, we observe that an adjacency hypermatrix of this AST generates a one-dimensional ternary algebra.


\smallskip

{\small
}

\normalsize

\section{Introduction}

A \textbf{classical association scheme} on a set \(\Omega\) is a partition of \(\Omega \times \Omega\) satisfying certain regularity properties \cite{eb}.
By providing a single framework with which to study several discrete structures, classical association schemes have proven useful in the study of combinatorial objects such as graphs, codes, and designs.

Introduced in \cite{mb} as a higher-dimensional analogue of a classical association scheme, an \textbf{association scheme on triples (AST)} on a set \(\Omega\) is a partition of \(\Omega \times \Omega \times \Omega\) satisfying certain symmetry requirements.
In that paper, the authors demonstrated several relationships between ASTs and known combinatorial structures, including classical association schemes, 2-designs, and two-graphs. 
Despite being introduced in 1990 and the various avenues open to exploration, the area is still in its nascency, with only a few papers published since its conception.
Several such papers appeared recently, including \cite{Praeger-2021-CirculantAST}, \cite{Suda-2025-ASTHadamard}, \cite{bb}, and \cite{Balmaceda-2023-survey-fusion}.

Still, little is known about the algebraic or combinatorial structure of ASTs. 
To this end, it is natural to consider the structural parameters of families of ASTs that admit well-understood automorphism groups.  
One such approach was taken in \cite{Praeger-2021-CirculantAST}, where ASTs under consideration were invariant under some transitive cyclic subgroup. 
In this paper, we instead continue an approach initiated in \cite{mb}, and furthered in \cite{bb}.
Similar to how classical association schemes arise from transitive group actions, \cite{mb} shows that ASTs arise naturally from two-transitive group actions. 
As there are few major families of two-transitive groups (as listed in \cite{dm}), and since this situation mirrors the familiar classical case, it is a natural goal to determine the structural constants of these ASTs.
The parameters of several families of such ASTs were obtained in \cite{bb}. 
Moreover, among the ASTs from the broadest classes of two-transitive groups, only the parameters of the ASTs from the Suzuki, Ree, and projective unitary groups remain undetermined.
%
We continue the work in \cite{bb} by leveraging the matrix representation of the unitary groups to describe explicitly the relations and intersection numbers of the ASTs these groups produce.
In order to afford convenient descriptions of the AST parameters, we differ slightly from the approach to the computations that was taken in \cite{bb}. 
The initial approach is similar, leveraging a convenient basis for the unitary space and the projective transformations. 
However, our method diverges when we phrase the necessary and sufficient conditions for an intersection number to be nonzero in terms of certain inclusions in the fixed field.
This yields the intersection numbers described in Theorem \ref{thm:main}.
In particular, we find that a one-dimensional ternary subalgebra is generated by a particular adjacency hypermatrix.

\medskip

\noindent{\textbf{Theorem \ref{thm:main}}.} {\it Let \(q\) be a prime power, $U(3,q^{2})$ be the unitary group of \(3\times 3\) matrices over \(GF(q^2)\), $\mathbb{P}C$ be the set of isotropic lines in the 3-dimensional vector space over \(GF(q^2)\), and $X$ be the AST from the $2$-transitive action of $U(3,q^{2})$ on $\mathbb{P}C$.
Then 
the intersection numbers of \(X\) are as listed in Table \ref{tab:intnum}.}




\section{Preliminaries}

\subsection{Association Schemes on Triples}

Following \cite{mb}, we define an association scheme on triples as follows.

\begin{definition}
Let $\Omega$ be a finite non-empty set with at least 3 elements. 
An \textbf{{association scheme on triples} (AST)} on $\Omega$ is a partition $\left\{R_{0},R_{1},R_{2},R_{3},\ldots,R_{m}\right\}$ of $\Omega\times\Omega\times\Omega$ such that $m\geq 4$ and the following conditions hold:
\begin{enumerate}
\item[i.] For any $i\in\left\{0,1,2,3,\ldots,m\right\}$, there exists a non-negative integer $n_{i}$ such that for any distinct elements $x$ and $y$ in $\Omega$, 
$$n_{i}=\left|\left\{\omega\in\Omega:(x,y,\omega)\in R_{i}\right\}\right|.$$
\item[ii.] For any $i,j,k,\ell\in\left\{0,1,2,3,\ldots,m\right\}$, there exists a non-negative integer $p_{ijk}^{\ell}$ such that for any $(x,y,z)\in R_{\ell}$,
$$p_{ijk}^{\ell}=\left|\left\{\omega\in\Omega:(\omega,y,z)\in R_{i},(x,\omega,z)\in R_{j},(x,y,\omega)\in R_{k}\right\}\right|.$$
\item[iii.] For any $i\in\left\{0,1,2,3,\ldots,m\right\}$ and any permutation $\sigma$ of the set $\left\{1,2,3\right\}$, there exists  $j\in\left\{0,1,2,3,\ldots,m\right\}$ such that
$$R_{j}=\left\{(x_{\sigma(1)},x_{\sigma(2)},x_{\sigma(3)}):(x_{1},x_{2},x_{3})\in R_{i}\right\}.$$
\item[iv.] 
The trivial relations $R_{0}$, $R_{1}$, $R_{2}$, and $R_{3}$ are the following:
\begin{align*}
 R_{0}&=\left\{(x,x,x):x\in\Omega\right\},\\
 R_{1}&=\left\{(y,x,x):x,y\in\Omega, x\neq y\right\},\\
 R_{2}&=\left\{(x,y,x):x,y\in\Omega, x\neq y\right\},\\
 R_{3}&=\left\{(x,x,y):x,y\in\Omega, x\neq y\right\}.
 \end{align*}
\end{enumerate}
\label{def:ast}
\end{definition}

The integer $n_i$ is called the {\bf third valency} (or simply {\bf valency}) of the ternary relation $R_{i}$. 
The integer $p_{ijk}^{\ell}$ is called the {\bf intersection number} with respect to the relations $R_{i},R_{j},R_{k}$ and $R_{\ell}$. For ease of discussion, define the set
$$\Gamma_{ijk}^{\ell}(x,y,z):=\left\{w\in \Omega:(w,y,z)\in R_{i},(x,w,z)\in R_{j},(x,y,w)\in R_{k}\right\},$$ for any indices $i,j,k,\ell\in\{0,\ldots,m\}$ and any $(x,y,z)\in R_{\ell}$.
By construction, $p_{ijk}^{\ell}=\left|\Gamma_{ijk}^{\ell}(x,y,z)\right|$.
The following remark provides some readily computed parameters.

\begin{remark}[Proposition \(2.7\), \cite{mb}]\label{rem_given-int-nums}
 Suppose $X=\left\{R_{0},R_{1},R_{2},R_{3},\ldots R_{m}\right\}$ is an arbitrary AST on a set $\Omega$ for some $m\geq 4$. 
 The following hold:
\begin{center}
\begin{tabular}{|c|} \hline
 $n_{0}=n_3=0$ \\ \hline
 $n_{1}=n_2=1$\\ \hline
 $p_{000}^{0}=1$ \\ \hline
 $p_{123}^{0}=\left|\Omega\right|-1$\\ \hline
 $p_{011}^{1}=p_{202}^{2}=p_{330}^{3}=1$\\ \hline
 $p_{132}^{1}=p_{321}^{2}=p_{213}^{3}=1$\\ \hline
$p_{ijk}^{1}=p_{ijk}^{2}=p_{ijk}^{3}=0$, \text{where} $i,j,k\geq 4$ \\ \hline
 $p_{3\ell 1}^{\ell}=p_{\ell 32}^{\ell}=p_{21\ell}^{\ell}=1$, \text{where} $\ell\geq 4$\\ \hline
\end{tabular} 
\end{center}   
\end{remark} 

Since several other intersection numbers are readily determined, the following intersection numbers are the ones of interest.
\begin{center}
\begin{tabular}{|c|} \hline
 $p_{1jk}^{1}$, $p_{i2k}^{2}$, $p_{ij3}^{3}$, where $i,j,k\geq 4$\\ \hline
 $p_{ijk}^{\ell}$, where $i,j,k,\ell\geq 4$ \\ \hline
\end{tabular} 
\end{center}

\subsection{Algebra of Hypermatrices}

To interpret the intersection numbers of an AST algebraically, we view ASTs as hypermatrix ternary algebras \cite{mb}. 
Let \(X=\{R_i\}_{i=0}^m\) be an AST on a set \(\Omega\). 
For each \(i\in\{0,\ldots,m\}\), we define an \(|\Omega|\times |\Omega|\times |\Omega|\) cubic hypermatrix \(A_i\) whose entries are indexed by \(\Omega\). 
The \((x,y,z)\)-entry of \(A_i\) is
\[(A_i)_{xyz}=
\begin{cases}
    1, & \text{if }(x,y,z)\in R_i, \text{ and }\\
    0, & \text{if }(x,y,z) \notin R_i.
\end{cases}\]
The linear space \(\operatorname{Span}_\mathbb{C}(\{A_i\}_{i=0}^m)\) of the adjacency hypermatrices is a complex vector space under scalar multiplication and entry-wise addition.
To equip the span with a ternary algebra structure, define for each $A,B,C\in \operatorname{Span}_\mathbb{C}(\{A_i\}_{i=0}^m)$ the product \(ABC\) given by $$(ABC)_{xyz}=\displaystyle\sum^{}_{w\in\Omega}{(A)_{wyz}(B)_{xwz}(C)_{xyw}}.$$

The theorem below shows that the intersection numbers of \(X\) are the structure constants of a ternary algebra.

\begin{theorem}[{Theorem \(1.4\) and Corollary 2.8, \cite{mb}}]\label{thm:ternary-subalgebra}
Let \(X=\{R_i\}_{i=0}^m\) be an AST with corresponding adjacency hypermatrices \(\{A_i\}_{i=0}^m\). 
Then \(\operatorname{Span}_\mathbb{C}(\{A_i\}_{i=0}^m)\) is a ternary algebra satisfying
\(A_{i}A_{j}A_{k}=\displaystyle\sum^{m}_{\ell=0}{p_{ijk}^{\ell}A_{\ell}},\)
for any $i,j,k\in\left\{0,1,\ldots,m\right\}$.
Moreover, \(\operatorname{Span}_\mathbb{C}(\{A_i\}_{i=4}^m)\) is a ternary subalgebra of \(\operatorname{Span}_\mathbb{C}(\{A_i\}_{i=0}^m)\).
\end{theorem}
In particular, Theorem \ref{thm:ternary-subalgebra} suggests that the most interesting intersection numbers are \(p_{ijk}^\ell\) for \(i,j,k,l\geq 4\).

\subsection{ASTs from Two-Transitive Groups}

Let $G$ be a group acting on a set $\Omega$. We say that $G$ is \textbf{$\mathbf{2}$-transitive} on $\Omega$ if for any $a,b,c,d\in\Omega$ with $a\neq b$ and $c\neq d$, there exists $x\in G$ such that $a^x=c$ and $b^x=d$. 
We write $a^G$ and $G_{a}$ for the \textbf{orbit} and the \textbf{stabilizer} of an element \(a\in G\), respectively.

We state and outline a proof of the following theorem from \cite{mb}, thereby including details on the valencies and intersection numbers.

\begin{theorem}[{Theorem \(4.1\), \cite{mb}}]
Let $G$ be a group acting $2$-transitively on a set $\Omega$, and let \(\Omega^3\) denote \(\Omega\times\Omega\times\Omega\). Define the action
\begin{align*}
  \Omega^{3}\times G&\longrightarrow \Omega^{3}\\  
(a,b,c)&\xmapsto{\phantom{..}x\phantom{..}}{}(a^{x},b^{x},c^{x})
\end{align*}
of \(G\) on \(\Omega^3\).
The set of orbits of this action on $\Omega^3$ is an AST on $\Omega$. 
\label{ref:thm1}
\end{theorem}

\begin{proof}[Outline of proof.]
 Suppose $\left\{R_{0},R_{1},R_{2},R_{3},\ldots,R_{m}\right\}$ is the set of distinct orbits of the action of $G$ on $\Omega^3$.
Since Conditions ($iii$) and ($iv$) of Definition \ref{def:ast} follow directly from the two-transitivity of \(G\), we may assume that $R_{0},R_{1},R_{2}$, and $R_{3}$ are the trivial relations described in Definition \ref{def:ast}.   
To verify Condition (\(i\)) of Definition \ref{def:ast}, take any \(a\neq b\) and \(c\neq d\).
Choose \(x\) such that $(a^{x},b^{x})=(c,d)$. 
We then obtain the bijection 
\begin{align*}
\left\{e\in \Omega:(a,b,e)\in R_{i}\right\}&\longrightarrow\left\{f\in \Omega:(c,d,f)\in R_{i}\right\}\\   
e&\longmapsto e^x.
\end{align*}
To verify condition ($ii$) of Definition \ref{def:ast}, let $i,j,k,\ell\in\left\{0,1,\ldots,m\right\}$. 
Let $(a,b,c)$ and $(x,y,z)$ be in $R_{\ell}$. 
There exists $g\in G$ such that $(x,y,z)^{g}=(a,b,c)$.
The map $w\mapsto w^g$ from $\Gamma_{ijk}^{\ell}(x,y,z)$ into $\Gamma_{ijk}^{\ell}(a,b,c)$ is a bijection.  
\end{proof}

\begin{remark}[Remark 2.10 \cite{bb}, Lemma 4.2 \cite{mb}]\label{rem:AST-valencies-and-2-pt-stab}
Let \(G\) be a two-transitive group acting on a set \(\Omega\), and \(X\) be the AST from \(G\) as constructed in Theorem \ref{ref:thm1}. 
The verification of Condition $(ii)$ of Definition \ref{def:ast} in the above proof shows that the nontrivial relations of \(X\) are in correspondence with the orbits of any two-point stabilizer. 
In fact, it also shows that the valencies are the sizes of these orbits.    
\end{remark}

\subsection{AST from the Unitary Group}
 Let \(q\) be a prime power, \(F=GF(q^2)\), and \(\overline{c}=c^q\) for each \(c\in F\). For any subfield \(K\) of \(F\), let \(K^\times\) denote the multiplicative \textbf{group of units} of \(K\).
Further, let $F_{0}=\left\{r\in GF(q^2):r^q=r\right\}$ be the order \(q\) subfield of \(F\), and let \(V=F\times F\times F\) be the usual 3-dimensional vector space over \(F\). 
Define \(B:V\times V\longrightarrow V\) by \[B(x,y)=x_1\overline{y_2}+x_2\overline{y_1}+x_3\overline{y_3},\]
for each \(x=(x_1,x_2,x_3),y=(y_1,y_2,y_3)\in V\).
Since \(B\) is linear in the first variable and $\overline{B(x,y)}=B(y,x)$ for any \(x,y\in V\), the mapping \(B\) is a \textbf{Hermitian form} on \(V\) \cite{g}. 
In fact, \(B\) is non-degenerate; i.e. if \(x\in V\) satisfies \(B(x,y)=0\) for all \(y\in V\), then \(x=0\).

Let $U(3,q^2)$ be the group of all linear automorphisms $\sigma:V\rightarrow V$ such that 
$$B(\sigma(x),\sigma(y))=B(x,y).$$
Then $U(3,q^2)$ is called the {\bf unitary group} of $V$. Letting \(Z(U(3,q^2))\) denote the center of \(U(3,q^2)\), the quotient group $PGU(3,q^2)=\frac{U(3,q^2)}{Z(U(3,q^2))}$ is called the {\bf projective unitary group}.

Let \(x\in V\) be a nonzero vector and let \([x]\) be the one-dimensional subspace of \(V\) spanned by \(x\). If \(B(x,x)=0\), then we say that \(x\) is an \textbf{isotropic vector} and that \([x]\) is an \textbf{isotropic line}.
The set \(\mathbb{P}C\) of all isotropic lines is called the {\bf projective Hermitian quadratic cone}.
 A subspace $W$ of $V$ is {\bf totally isotropic} if every nonzero vector in $W$ is isotropic. 
 The dimension of a maximal totally isotropic subspace is called the {\bf Witt index} of $V$. 
 It can be shown that \(V\) has a Witt index of \(1\), and that \(V\) has a basis of vectors \(u,v,w\in V\) satisfying \(B(u,v)=1=B(w,w)\) and \(B(u,u)=0=B(v,v)\). 
 We fix and utilize such a basis \(\{u,v,w\}\) throughout the article.
 

\section{Parameters of the AST from the Unitary Group}

The group $U(3,q^{2})$ acts two-transitively on $\mathbb{P}C$ via $\sigma[x]=[\sigma(x)]$ for $\sigma\in U(3,q^{2})$ \cite{g}. 
It can be shown that the kernel of this action is $Z(U(3,q^{2}))$, so that $PGU(3,q^{2})$ acts faithfully and two-transitively on $\mathbb{P}C$. 
Let \(\mathbb{P}C^3=\mathbb{P}C\times \mathbb{P}C\times \mathbb{P}C\) and define the following action
\begin{align*}
     \mathbb{P}C^3\times U(3,q^{2})&\longrightarrow \mathbb{P}C^3 \\
    ([x],[y],[z])&\xmapsto{\phantom{..}\sigma\phantom{..}}{}([\sigma(x)],[\sigma(y)],[\sigma(z)])
\end{align*}
of \(U(3,q^2)\) on \(\mathbb{P}C^3\).
By Theorem \ref{ref:thm1}, the set $X$ of orbits of this action forms an association scheme on triples AST on \(\mathbb{P}C\).
We compute the intersection numbers of this AST in this section. 
Through these computations, we recover the number of relations and the valencies of this AST, initially obtained in \cite{bb}. 

\begin{theorem}[Table \(1\), \cite{bb}]
Let $X$ be the AST from the $2$-transitive action of $U(3,q^{2})$. Then the number of relations of $X$ is $q+5$.
\end{theorem}

For our purposes, we require explicit forms for the nontrivial relations of this AST. 
To accomplish this, first define the sets 
\[\Omega_{0}:=\left\{[au+v]\in\mathbb{P}C:a\in F^\times,a+\overline{a}=0\right\},\] \
and
\[\Omega':=\left\{[bu+v+cw]\in\mathbb{P}C:b,c\in F,\; b+\overline{b}+c\overline{c}=0\neq b+\overline{b}\right\}.\]
It follows that \(\mathbb{P}C\) is the disjoint union of the sets \(\left\{[u],[v]\right\}\), \(\Omega_{0}\), and \(\Omega'\). 

By Remark \ref{rem:AST-valencies-and-2-pt-stab}, there is a correspondence between the orbits of $U(3,q^2)$ on $\mathbb{P}C^{3}$ and the orbits of $U(3,q^2)_{[u],[v]}$ on $\mathbb{P}C\setminus\left\{[u],[v]\right\}$. Since it can be verified that $U(3,q^2)_{[u],[v]}$ acts transitively on $\Omega_{0}$, it remains to consider the orbits of $U(3,q^{2})$ on $\Omega'$.
These are described in Proposition \ref{prop-UnitaryASTorbits}, affording us a workable description of the nontrivial relations.

\begin{proposition}\label{prop-UnitaryASTorbits}
The group $U(3,q^{2})_{[u],[v]}$ acts on $\Omega'$, and 
this action has \(q\) orbits.
\label{ref:prop1}
\end{proposition}

\begin{proof} 
%
To obtain the number of orbits of \(U(3,q^{2})_{[u],[v]}\) on \(\Omega'\), we begin by determining \(|\Omega'|\). 
Since the trace map \(T:F\rightarrow F_{0}\) given by \(T(x)=x+\overline{x}\) is an additive epimorphism, we have
\begin{equation}\label{eq:orbit_type1_size}
    |\operatorname{Ker}{T}|=\frac{|F|}{|F_{0}|}=\frac{q^{2}}{q}=q.
\end{equation}
Moreover,
\[F\setminus\operatorname{Ker}{T} = \{b\in F^\times : b+\overline{b}\neq 0\}.\]
Since the norm map \(N:F^\times\rightarrow F_{0}^\times\) with \(N(r)=r\overline{r}\) is a multiplicative group epimorphism, we deduce that \(|\operatorname{Ker}{N}|=q+1\). For every 
\(b\in F\setminus\operatorname{Ker}{T}\), the set of all elements \(c\in F\) such that \(b+\overline{b}+c\overline{c}=0\) forms a coset of \(\operatorname{Ker}{N}\). Hence, there are exactly \(q+1\) such elements \(c\in F\) satisfying \(b+\overline{b}+c\overline{c}=0\). 

Thus,
$$|\Omega'|=|F\setminus\operatorname{Ker}{T}|(q+1)=(q^{2}-q)(q+1)=q(q^{2}-1).$$

Proceeding, observe that for any $[bu+v+cw],[b_{1}u+v+c_{1}w]\in \Omega'$, we have $[bu+v+cw]\in [b_{1}u+v+c_{1}w]^{U(3,q^{2})_{[u],[v]}}$ if and only if $(F_{0}^\times)b=(F_{0}^\times)b_{1}$. 
It follows that any orbit contains exactly 
\begin{equation}\label{eq:orbit_type2_size}
    |F_0^\times|(q+1)=(q-1)(q+1)=q^{2}-1
\end{equation}
isotropic lines. %
%
%
%
Therefore,
$$\frac{|\Omega'|}{q^{2}-1}=\frac{q(q^{2}-1)}{q^{2}-1}=q$$ 
is the number of orbits of $U(3,q^{2})_{[u],[v]}$ on $\Omega'$.
\end{proof}

The proof of Proposition \ref{ref:prop1} 
implies that the orbits of $U(3,q^2)_{[u],[v]}$ on $\Omega'$ correspond to the cosets of $F_0^\times$ in $F^\times$. 
To enumerate the relations in \(X\), let \(\{b_5,\ldots,b_{q+4}\}\) 
be a complete set of coset representatives for $F^\times/F_0^\times$.
Further, fix for each \(i\in \{5,\ldots,q+4\}\) an element \(c_i\in F\) such that \(b_i+\overline{b_i}+c_i\overline{c_i}=0\).
Lastly, fix a nonzero \(a\in F\) 
such that \(a+\overline{a}=0\).
The AST \(X\) has two kinds of nontrivial relations.
There is one of the first type, namely 
\begin{equation}\label{eq:relation_4}
R_{4}=([u],[v],[au+v])^{U(3,q^{2})}.  
\end{equation}
The remaining \(q\) relations of \(X\) are 
\begin{equation}\label{eq:relation_i}
R_{i} = ([u],[v],[b_{i}u+v+c_{i}w])^{U(3,q^{2})},    
\end{equation} 
where \(i\in\{5,\ldots,q+4\}\).
The valencies of the nontrivial relations are then obtained as the following consequence of Equations (\ref{eq:orbit_type1_size}) and (\ref{eq:orbit_type2_size}).


\begin{corollary}[Table \(1\), \cite{bb}]
Let $X$ be the AST from the $2$-transitive action of $U(3,q^{2})$. 
Then the nontrivial valencies are $n_{4}=q-1$ and $n_{i}=q^{2}-1$ for $i\geq 5$.    
\end{corollary}

We may sharpen the description of the relations of the AST from \(U(3,q^2)\) when the underlying field has even characteristic.
\begin{remark}
Let \(q\) be a power of two, \(X\) be the AST from the action of \(U(3,q^2)\), and \(F\) be the field of cardinality \(q^2\). 
Fix a primitive element \(\alpha\) of \(F\) and let \(T:F\rightarrow F_0\) be given by \(T(x)=x+\overline{x}\). Since \(F_0^\times \subseteq \operatorname{Ker}T\), we may choose \(a=\alpha^{q+1}\) and \(b_i=\alpha^{(q+1)+i}\) for \(i\in\{4,\ldots,q+4\}\) in Equations (\ref{eq:relation_4}) and (\ref{eq:relation_i}).
%
%
\end{remark}

We proceed with the intersection numbers of the AST from \(U(3,q^2)\).

\begin{theorem}\label{thm:main}
Let \(q\) be a prime power, $U(3,q^{2})$ be the unitary group of \(3\times 3\) matrices over \(GF(q^2)\), $\mathbb{P}C$ be the set of isotropic lines in the 3-dimensional vector space over \(GF(q^2)\), and $X$ be the AST from the $2$-transitive action of $U(3,q^{2})$ on $\mathbb{P}C$.
For \(5\leq i,j,k,\ell\leq q+4\), the intersection number \(p_{ijk}^\ell\) of \(X\) is as listed in Table \ref{tab:intnum}.

\begin{center}
\begin{table}[H]
\begin{center}
\scalebox{.706}{
{
\begin{tabular}{|c|l|}\hline
1 & $p_{144}^{1}=p_{424}^{2}=p_{443}^{3}=q-1$ \\ \hline
2 & $p_{14i}^{1}=p_{1i4}^{1}=p_{42i}^{2}=p_{i24}^{2}=p_{4i3}^{3}=p_{i43}^{3}=0$ \\ \hline
3 & $p_{1ji}^{1}=
\begin{cases}
  q^{2}-1 & \text{if } R_{j}=([u],[v],[\overline{b_i}u+v+c_iw])^{U(3,q^2)},\\
			
      0 & \text{otherwise}.\\  
\end{cases}$ 
\\ \hline
4 & $p_{j2i}^{2}=
\begin{cases}
      q^{2}-1 & \text{ if }R_{j}=([u],[v],[b_i^{-1}u+v+(b_i^{-1}c_i)w])^{U(3,q^2)},\\
      0 & \text{otherwise}.\\
\end{cases}$ \\ \hline
5 & $p_{ji3}^{3}=
\begin{cases}
      q^{2}-1 & \text{ if }[\overline{b_i}u+v+c_iw]^{U(3,q^2)_{[u],[v]}}=[\overline{b_{j}^{-1}}u+v+(b_{j}^{-1}c_{j})w]^{U(3,q^2)_{[u],[v]}},\\
      0 & \text{otherwise}.\\
\end{cases}$\\ \hline
& For items 6 -- 8, fix $([u],[v],[au+v])\in R_{4}$.\\ \hline
6 & $p_{444}^{4}=q-2$ \\ \hline
7 & $p_{i44}^{4}=p_{4j4}^{4}=p_{44k}^{4}=p_{ij4}^{4}=p_{i4k}^{4}=p_{4jk}^{4}=0$ \\ \hline
8 & $p_{ijk}^{4}=
\begin{cases}
      q+1 & \text{ if }b_{j}(\overline{b}+a)^{-1}\in F_{0}^\times\text{ and }b_{i}b(\overline{b}+a)a^{-1}\in F_{0}^\times,\text{ where }([u],[v],[bu+v+cw])\in R_{k},\\
      0 & \text{otherwise}.\\
\end{cases}$ \\ \hline
& For items 9 -- 13, fix $([u],[v],[b_{\ell}u+v+c_{\ell}w])\in R_{\ell}$. \\ \hline
 9 & $p_{i44}^{\ell}=p_{4j4}^{\ell}=p_{44k}^{\ell}=p_{444}^{\ell}=0$ \\ \hline
10 & $p_{4jk}^{\ell}=
\begin{cases}
      1 & \text{ if }b_{j}(\overline{b}+b_{\ell}+\overline{c}c_{\ell})^{-1}\in F_{0}^\times\text{ and }b_{\ell}c=bc_{\ell}, \text{ where }([u],[v],[bu+v+cw])\in R_{k},\\
      0 & \text{otherwise}\\
\end{cases}$ \\ \hline
11 & $p_{i4k}^{\ell}=
\begin{cases}
      1 & \text{ if }b_{i}bb_{\ell}^{-1}(\overline{b}+b_{\ell}+\overline{c}c_{\ell})\in F_{0}^{\times}\text{ and }c=c_{\ell},\text{ where }([u],[v],[bu+v+cw])\in R_{k},\\
      0 & \text{otherwise}.\\
\end{cases}$ \\ \hline
12 & $p_{ij4}^{\ell}=
\begin{cases}
      1 & \text{ if }b_{i}(\overline{g}+b_{\ell})gb_{\ell}^{-1}\in F_{0}^\times\text{ and }b_{j}(\overline{g}+b_{\ell})^{-1}\in F_{0}^\times, \text{ where }([u],[v],[gu+v])\in R_{4},\\
      0 & \text{otherwise}.\\
\end{cases}$ \\ \hline
13 & $p_{ijk}^{\ell}=\left|\left\{
([u],[v],[bu+v+cw])\in R_{k}:b_{i}bb_{\ell}^{-1}(\overline{b}+b_{\ell}+\overline{c}c_{\ell}),b_{j}(\overline{b}+b_{\ell}+\overline{c}c_{\ell})^{-1}\in F_{0}^\times 
\right\} \right|$ \\ 
\hline
\end{tabular}
}

}
\end{center}
\caption{Intersection numbers \(p_{ijk}^\ell\) of the AST from $U(3,q^{2})$, where $5\leq i,j,k,\ell\leq q+4$.} 
\label{tab:intnum}

\end{table}    
\end{center}

\end{theorem}

\begin{proof} 
We only show that line 8 of Table \ref{tab:intnum} holds; the other lines of Table \ref{tab:intnum} can be shown by applying a similar approach. 
Let \(([u],[v],[au+v])\in R_4\). 
Observe that \(p_{ijk}^4\neq 0\) if and only if
\begin{equation}\label{eq:case-8}
\Gamma_{ijk}^4([u],[v],[au+v])\neq \emptyset.
\end{equation}
This is equivalent to the existence of \([\zeta]=[bu+v+cw]\) in \(([b_ku + v + c_kw])^{U(3,q^2)_{[u],[v]}}\) satisfying the following conditions.
\begin{enumerate}
	\item[C1.] There exist \(\sigma \in U(3,q^2)\) and \(r,s,t \in F^\times\) such that 
	\begin{equation}\label{eq:Case8-2}
		\sigma(u,v,b_i u + v +c_i w) = (r(bu+v+cw),sv,t(au+v)).
	\end{equation}
	\item[C2.] There exist \(\tau \in U(3,q^2)\) and \(r_1,s_1,t_1 \in F^\times\) such that 
	\begin{equation}\label{eq:Case8-3}
		\tau(u,v,b_j u + v +c_j w) = (r_1 u, s_1(bu+v+cw),t_1 (au+v)).
	\end{equation}  
\end{enumerate}
  Since \(\tau\in GL(3,q^2)\) is in \(U(3,q^2)\) if and only if \(B(\tau(x),\tau(y)) = B(x,y)\) for \(x,y\in\{u,v,w\}\), Condition C2 is equivalent to 
  \begin{equation}\label{eq:Case8-3-1}
0=\overline{s_1}r_1-1=s_1-t_1=c\overline{c} s_1 +b_jr_1 + bs_1 - at_1,
  \end{equation}
  for some \(r_1,s_1,t_1\in F^\times\). 
 In turn, this is equivalent to 
 \begin{equation}\label{eq:Case8-3-2}
 	b_j(\overline{b}+a)^{-1}\in F_0^\times.
 \end{equation}
  Indeed, Equation (\ref{eq:Case8-3-1}) implies \(b_j(\overline{b}+a)^{-1}\in F_0^\times\). 
 Conversely, if \(b_j(\overline{b}+a)^{-1}\in F_0^\times\), then we may choose \(s_1\in F^\times\) such that \(b_j=s_1 \overline{s_1}(\overline{b}+a)\).
 Equation (\ref{eq:Case8-3-1}) is then satisfied by taking \(t_1=s_1\) and \(r_1=\overline{s_1}^{-1}\). 

Similarly, Condition C1 holds if and only if 
\begin{equation}\label{eq:Case8-2-1}
b_i b(\overline{b}+a)a^{-1}\in F_0^\times.
\end{equation}
%
It remains to show that if \(p_{ijk}^4\neq 0\), then there are exactly \(q+1\) elements \([\zeta]\) in the orbit \(([b_ku + v + c_kw])^{U(3,q^2)_{[u],[v]}}\) such that Equations (\ref{eq:Case8-3-2}) and (\ref{eq:Case8-2-1}) are true.
To see this, fix \([\zeta]=[bu+v+cw] \in \Gamma_{ijk}^4([u],[v],[au+v]).\)
Recall that the elements of \(([b_ku + v + c_kw])^{U(3,q^2)_{[u],[v]}}\) are \([du+v+ew]\), where \(d \in b F_0^\times\), and \(e\in F^\times\) satisfies \(d+\overline{d}+e\overline{e}=0\).
For each of the \(q+1\) elements \(e\in F^\times\) satisfying \(e\overline{e}=c\overline{c}\), the element \([bu+v+ew]\) satisfies Equations (\ref{eq:Case8-3-2}) and (\ref{eq:Case8-2-1}).  
This accounts for \(q+1\) elements of \(\Gamma_{ijk}^4([u],[v],[au+v])\).
To complete the proof, we show that \([du+v+ew]\notin \Gamma_{ijk}^4([u],[v],[au+v])\) if \(b\neq d\). 
Indeed, let \(d=by\) for some \(1\neq y \in F_0^\times \). 
If \([du+v+ew] \in  \Gamma_{ijk}^4([u],[v],[au+v])\), then Equation (\ref{eq:Case8-2-1}) yields \(b_id(\overline{d}+a)a^{-1} \in F_0^\times\). 
Since \(b_ib(\overline{b}+a)a^{-1}\) is also in \(F_0^\times\), we obtain 
\begin{equation}\label{eq:case8-counting}
  y^{-1}(\overline{b}+a)(a+\overline{b}y)^{-1}=b_ib(\overline{b}+a)a^{-1}(b_id(\overline{d}+a)a^{-1})^{-1} =z,   
\end{equation}for some \(z\in F_0^\times\).
Since \(y\neq 1\), Equation (\ref{eq:case8-counting}) implies \(yz\neq 1\).  
Consequently, Equation (\ref{eq:case8-counting}) yields \(\frac{a}{\overline{b}}=\frac{1-y^2z}{zy-1}=\frac{\overline{a}}{b}=-\frac{a}{b}\).
This implies \(b+\overline{b}=0\), which is impossible.
\end{proof}


As a consequence of Theorem \ref{thm:main}, we obtain a minimal ternary subalgebra generated by the hypermatrix \(A_4\) of the AST from \(U(3,q^{2})\).

\begin{corollary} 
Let $X=\left\{R_{0},R_{1},R_{2},R_{3},\ldots,R_{q+4}\right\}$ be the AST from the $2$-transitive action of $U(3,q^{2})$ on $\mathbb{P}C$, where the relations are as given in Equations (\ref{eq:relation_4}) and (\ref{eq:relation_i}). 
If $\left\{A_{0},A_{1},A_{2},A_{3},\ldots A_{q+4}\right\}$ is the set of corresponding adjacency hypermatrices of \(X\), 
then $\operatorname{Span}\left\{A_{4}\right\}$ is a ternary subalgebra of 
$\operatorname{Span}\left\{A_{4},\ldots A_{q+4}\right\}$.  
\end{corollary}

\begin{proof} By lines \(6\) and \(9\) in Table \ref{tab:intnum}, we have
\[A_{4}A_{4}A_{4}=\sum_{\ell=4}^{q+4}{p_{444}^{\ell}A_{\ell}}=p_{444}^{4}A_{4}+\sum_{\ell=5}^{q+4}{p_{444}^{\ell}A_{\ell}}=(q-2)A_{4}. \qedhere\]
\end{proof}


When the characteristic of the underlying field is two, several intersection numbers vanish.

\begin{corollary} Let \(q\) be a power of two and $X$ be the AST from the $2$-transitive action of $U(3,q^{2})$ on $\mathbb{P}C$. 
If the relations of \(X\) are as given in Equations (\ref{eq:relation_4}) and (\ref{eq:relation_i}), then
$$p_{k\ell k}^{\ell}=p_{\ell kk}^{\ell}=p_{kk\ell}^{\ell}=p_{\ell\ell\ell}^{\ell}=0,$$ 
whenever $k,\ell\in\{5,\ldots,m\}$ are distinct. 
Moreover, if \(5\leq i\leq q+4\) then $\operatorname{Span}\left\{A_{i}\right\}$ cannot be a nontrivial subalgebra of $\operatorname{Span}\left\{A_{4},\ldots A_{q+4}\right\}$. 
\end{corollary}
\begin{proof} We only show that $p_{k\ell k}^{\ell}=0$ if $k,\ell\in\{5,\ldots,m\}$ are distinct. 
Assume on the contrary that $p_{k\ell k}^{\ell}\neq 0$. 
Choose any $([u],[v],[bu+v+cw])\in R_{\ell}$ and suppose $([u],[v],[hu+v+ew])\in R_{k}$ satisfies the conditions of Equation $13$ of Table \ref{tab:intnum}. As a consequence, we have
$$hhb^{-1}b\in F_{0}.$$
It follows that $h^{2}\in F_{0}$. Since $h+\overline{h}+e\overline{e}=0$, we obtain
\begin{align*}
h+e\overline{e}&=-\overline{h}, \text{ and }\\
h&=-\overline{h}-e\overline{e}.
\end{align*}
Multiplying the corresponding sides of the above equations yields $h^{2}+he\overline{e}=\overline{h}^{2}+\overline{h}e\overline{e}$. Since $h^{2}=\overline{h}^{2}$, it follows that $h=\overline{h}$. That is, $h\in F_{0}^\times$. Since $\operatorname{char}F=2$ and $k\geq 5$, we obtain a contradiction.\\
The final statement holds since \(p_{iii}^{i}=0\) for \(5\leq i\leq q+4\)
\end{proof}

\end{document}